\documentclass[10pt,a4paper,twoside]{article}
\usepackage{amsfonts,amssymb,amsmath,amsthm,graphics,titlesec,fancyhdr}

\renewcommand\thesection{\arabic{section}.}
\renewcommand\thesubsection{\thesection\arabic{subsection}.}

\titleformat{\section}
  {\normalfont\large\bfseries}{\thesection}{0.5em}{}
\newcommand{\addperiod}[1]{#1.}
\titleformat{\subsection}[runin]
  {\normalfont\bfseries}{\thesubsection}{0.25em}{\addperiod}

\theoremstyle{plain}
\newtheorem*{definition}{\bf Definition}
\newtheorem*{remark}{\it Remark}
\newtheorem{proposition}{\bf Proposition}

\newcommand{\der}[2]{\frac{\partial#1}{\partial#2}}
\newcommand{\dder}[3]{\frac{\partial^2#1}{\partial#2 \partial#3} }
\newcommand{\cC}{\mathcal{C}}
\newcommand{\cD}{\mathcal{D}}
\newcommand{\cM}{\mathcal{M}}
\newcommand{\cR}{\mathcal{R}}
\newcommand{\bR}[1]{\mathbb{R}^{#1}}
\newcommand{\s}{\sigma}
\newcommand\mf[1]{\hat{#1}}  

\begin{document}

\title{\Large \bf 
Conservation of energy and momenta \\
in nonholonomic systems with affine constraints\footnote{
This work is part of the research projects {\it
Symmetries and integrability of nonholonomic
mechanical systems} of the University of Padova and PRIN {\it Teorie
geometriche e analitiche dei sistemi Hamiltoniani in dimensioni finite
e infinite}.}
}

\author{\sc Francesco Fass\`o\footnote{\footnotesize Universit\`a di Padova,
Dipartimento di Matematica, Via Trieste 63, 35121 Padova, Italy.
Email: {\tt fasso@math.unipd.it} } 
\  and
Nicola Sansonetto\footnote{\footnotesize Universit\`a di Padova,
Dipartimento di Matematica, Via Trieste 63, 35121 Padova, Italy.
Email: {\tt sanson@math.unipd.it} }\ \footnote
{Supported by the Research Project
{\it Symmetries and integrability of nonholonomic
mechanical systems} of the University of Padova.}
}

\date{}
\maketitle
\centerline{\small (\today)}

{\small 
\begin{abstract}
\noindent
We characterize the conditions for the conservation of
the energy and of the components of the momentum maps of lifted actions, and
of their `gauge-like' generalizations, in
time-independent nonholonomic mechanical systems with affine
constraints. These conditions involve geometrical and mechanical
properties of the system, and are codified in the  so-called
reaction-annihilator  distribution. 

\vskip3mm
{\scriptsize
\noindent
{\bf Keywords:} Nonholonomic mechanical systems, Conservation of
energy, Reaction-annihilator \break
distribution, Gauge momenta, Nonholonomic Noether theorem. 

\vskip1mm
\noindent
{\bf MSC:} 70F25, 37J60, 37J15, 70E18
}

\end{abstract}
}

\section{Introduction}

In this paper we study the conservation of the energy and of the
components of the momentum maps of lifted actions---and of their
`gauge' generalizations introduced in \cite{BGM}---in time-independent
mechanical
systems with ideal nonholonomic  constraints that are affine
functions of the velocities. The conservation of both types of
functions is affected by the reaction forces exerted by the
nonholonomic constraint,  but there is a difference between them. The
conservation of energy depends on certain properties of the
nonhomogeneous term of the constraint distribution, and thus differs
from the case of nonholonomic systems with constraints that are
linear functions of the velocities. Instead, the conservation of
momenta and gauge momenta is essentially the same as that in the case
of linear constraints, which has been extensively studied
\cite{bates-sniatycki,marle95,FRS,FGS2008}. Therefore, we will
mainly focus on the conservation of energy. (If the Lagrangian is not
quadratic kinetic energy minus positional potential energy, then the
function we call energy is, properly, the Jacobi integral).

It is well known that the energy is conserved in nonholonomic
mechanical systems with constraints that are linear in the velocities
\cite{pars, NF}, and that it is not always conserved if the
constraints are affine, or more generally nonlinear, in the
velocities. More than specifically in the case of affine constraints,
the issue of energy conservation has received extensive consideration
in the case of general nonlinear constraints \cite{BKMM, spagnoli,
sniatycki98, marle2003,
kobayashi-oliva2003,BM2008,bates-jim,zampieri}. The equation of
balance of the energy shows that the energy is conserved if and only
if the reaction forces exerted by the constraint do not do any work
on the constrained motions. The quoted references deduce from this
some sufficient conditions for the conservation of energy, such as
the tangency of the Liouville vector field to the constraint manifold
(as e.g. in \cite{marle2003,spagnoli}) or the fact that the
constraint is a homogeneous function of the velocities (as e.g. in
\cite{BM2008,bates-jim}). However, when particularized to the case of
affine constraints, as e.g. in  \cite{marle2003,
kobayashi-oliva2003}, all these sufficient conditions reduce to the
linearity of the constraints. It seems, therefore, that conservation
of energy for nonholonomic  systems with affine constraints is
presently not understood. The main purpose of this paper is to remedy
this lacune, by identifying the properties of a nonholonomic
mechanical system  with affine constraints that determine the
conservation of its energy. 

At the basis of our approach is the fact that, under the hypothesis of
ideality of the constraints (namely, d'Alembert principle), the
reaction force that the constraint exerts on the system is a known
function of the kinematic state of the system. This function depends
on properties of the system that are both of geometric nature (the
nonholonomic constraint) and of mechanical nature (the mass
distribution and the active forces\footnote{By `active forces' we
mean the forces that act on the system and are not reaction forces;
in part of the literature they are called `external forces'.}). The
inspection of this function reveals that, for a given system, the set
of all reaction forces exerted by the constraints on the constrained
motions might be (and typically is) {\it smaller} than the set of all
reaction forces that satisfy the condition of ideality. The reason is
that for a given system the active forces that act on the system are
fixed, while the notion of ideality makes reference to all possible
active forces that might possibly act on the system (see \cite{FS2009}
for a discussion of this fact). Consequently, any vector that
annihilates the linear part of the constraint may be an ideal
reaction force, but for a given system, the class of reaction forces
actually exerted by the constraint may be a subset of this
annihilator.

Therefore, the properties of a nonholonomic system that are
influenced by the reaction forces exerted on constrained motions may
depend, in a complicate way, on the geometric and mechanical
properties of the system, including the active forces. These
properties can be codified by a distribution on the configuration
manifold, which is called the {\it reaction-annihilator
distribution}. This distribution was introduced in \cite{FRS} in
connection with the conservation of the momentum map of lifted
actions in nonholonomic systems with linear constraints, and was
further used in \cite{FGS2008,FGS2009,FS2009, crampin,FGS2012,jotz2}. 

We will show that a necessary and sufficient condition for energy
conservation in nonholonomic mechanical systems with affine
constraints is that the nonhomogeneous term of the constraint is a
section of the reaction-annihilator distribution. This clarifies in a
quantitative, computable way why energy conservation is not a
property of purely geometric type and, in particular, how it depends
on the active forces that act on the system: changing just the
(conservative) active forces that act on a given system (same
constraints, same kinetic energy) may destroy, or restore, the
conservation of energy. We will illustrate these behaviours on some
examples.

We recall the basic facts about nonholonomic mechanical systems with
affine constraints, and introduce the reaction-annihilator
distribution for these systems, in Section~2. Energy conservation is
studied in Section~3 and the conservation of momenta and gauge
momenta of lifted actions is concisely studied in Section 4. A short
Conclusion follows, where we stress the importance of exploiting the
knowledge of the reaction forces in the study of nonholonomic
systems. In the Appendix we derive the expression of the reaction
force as function of the kinematic state. 

Throughout the paper all manifolds and maps are smooth and all vector
fields are assumed to be complete. For simplicity we restrict our
consideration to time-independent systems. For introductions to nonholonomic
mechanics see e.g. \hbox{\cite{pars,NF, pagani91, cortes,
marle2003,benenti,CDS}}. 

Lastly, we mention that in certain nonholonomic mechanical systems
with affine constraints,
even if the energy is not conserved, there may exist
a modification of it (which
may be interpreted as the energy of the system in a moving reference
frame and has therefore been called a `moving energy') that is
conserved \cite{FS2015}.

\section{Affine constraints and the reaction-annihilator distribution}

\subsection{Nonholonomic systems with ideal affine constraints}
Since affine constraints appear typically
in problems of rigid bodies that roll on moving surfaces, it is
appropriate to work on a phase space that is a manifold. However, for
simplicity we will resort wherever possible to a coordinate
description.\footnote{In the sequel, symbols with a hat denote global
objects and the same symbols without the hat their local,
coordinate representatives.} Moreover, because of the possible
presence of moving holonomic constraints in this type of systems, it
is natural to allow for the presence of gyrostatic terms in the
Lagrangian, that may come either from the use of non-inertial frames
or from the use of moving coordinates. We assume however that the
system is time-independent, as is it typically happens if the bodies and the
surface have suitable symmetries and the latter moves at uniform
speed. 

Our starting point is thus a Lagrangian system with $n$-dimensional
configuration manifold $\mf{Q}$ and Lagrangian
$\mf{L}:T\mf{Q}\to\bR{}$, that describes a mechanical system
subject to ideal holonomic constraints. We assume that the
Lagrangian has the mechanical form 
\begin{equation}\label{hatL}
  \mf{L}=\mf{T}-\mf{b} - \mf{V}\circ\pi
\end{equation}
where $\mf{T}$ is a positive definite quadratic form on $T\mf{Q}$,
$\mf{b}$ is a 1-form on $\mf{Q}$ regarded as a function on $T\mf{Q}$,
$\mf V$ a function on $Q$ and $\pi: T\mf{Q} \to \mf{Q}$ is the tangent
bundle projection.  Following e.g. \cite{marle2003} we write Lagrange
equations as  $[\mf L]=0$, where $[\mf L]$ may be regarded as a
1-form on $\mf Q$, whose coordinate expression is the well known
$\frac d{dt}\der L {\dot q} -\der L q$.

We add now the nonholonomic constraint that, at each point $\mf{q}\in
\mf{Q}$, the velocities of the system belong to an affine subspace
$\mf{\cM}_{\mf{q}}$ of the tangent space $T_{\mf q}\mf{Q}$.
Specifically, we assume that there are a nonintegrable distribution
$\mf{\cD}$ on $\mf{Q}$ of constant rank $r$, with $1<r<n$, and a
vector field $\mf{\xi}$ on $\mf{Q}$ such that, at each point
$\mf{q}\in \mf{Q}$,
$$
  \mf{\cM}_{\mf q}=\mf{\xi}(\mf{q})+\mf{\cD}_{\mf q} \,.
$$
Clearly, the vector field $\mf{\xi}$ is defined up to a section
of~$\mf{\cD}$. The affine distribution $\mf{\cM}$ with fibers
$\mf{\cM}_{\mf q}$ may also be regarded as a submanifold
$\mf{M}\subset T\mf{Q}$ of dimension $n+r$, which is actually an
affine subbundle of $T\mf{Q}$ of rank $r$, and is called the {\it
constraint manifold}. The case of linear constraints is recovered
when the vector field $\mf{\xi}$ is a section of the distribution
$\mf{\cD}$, since then $\mf{\cM}=\mf{\cD}$.

We assume that the nonholonomic constraint is `ideal', namely, that it
satisfies d'Alembert principle. This means that, when the system is
in a configuration $\mf{q}\in \mf{Q}$, the set of reaction forces
that the nonholonomic constraint is capable of exerting  coincides
with the annihilator $\mf{\cD}_{\mf q}^\circ$ of $\mf{\cD}_{\mf q}$ 
(see e.g. \cite{pagani91,marle2003}). Under this hypothesis there is
a unique function $\mf{R}_{\mf{L},\mf{M}}:\mf{M}\to\mf{\cD}^\circ$,
namely a function that associates an ideal reaction force $\mf
R_{\mf L,\mf M}$ to each constrained kinematic state
$\mf{v}_{\mf q}\in \mf{M}$, which has the property that the restriction to
$\mf{M}$ of Lagrange equations with the reaction forces,
\begin{equation}\label{EqLagrWithRF}
  [\mf{L}]\big|_{\mf M}  = \mf{R}_{\mf{L},\mf{M}} \,,
\end{equation}
defines a dynamical system on $\mf M$ (that is, a vector field on
$\mf{M}$). For completeness, we give a proof of this fact in the
Appendix.

\begin{definition} Assume that $\mf L:T\mf Q\to\bR{}$ is as in
(\ref{hatL}) and that $\mf M$ is an affine subbundle of $T\mf Q$.
The {\rm nonholonomic mechanical system with affine
constraints} $(\mf{L},\mf{Q},\mf{M})$ is the
dynamical system defined by equation (\ref{EqLagrWithRF})
on $\mf{M}$. 
\end{definition}

\subsection{Coordinate description}

We consider now a system of local coordinates $q$ on $\mf{Q}$, with
domain $Q\subseteq\bR n$, and lift them to bundle coordinates
$(q,\dot q)\in Q\times\bR n$ in $T\mf{Q}$. We write the local
representative of the Lagrangian
$\mf{L}=\mf{T}+\mf{b}-\mf{V}\circ\pi$ as
\begin{equation}\label{L}
  L(q,\dot q) = \frac12\dot q\cdot
  A(q)\dot q+b(q)\cdot q - V(q)
\end{equation}
with $A(q)$ an $n\times n$
symmetric nonsingular matrix and $b(q)\in\bR n$

The fibers of the local representative $\cD$ of the distribution
$\mf{\cD}$ can be described as the kernel of a $q$-dependent $k\times
n$ matrix $S(q)$ that has everywhere rank $k$, with $k=n-r$: 
$$
   \cD_q=\{\dot q\in T_qQ =\bR n\,:\; S(q)\dot q=0\} \,.
$$
The matrix $S$ is not uniquely defined. However, if $S_1$ and $S_2$
are any two possible choices of it, then from $\ker S_1 = \ker S_2$
it follows that there exists a $q$-dependent $k\times k$ nonsingular
matrix $P(q)$ such that $S_2 = P S_1$.

Let now $\cM$ be the local representative of $\mf{\cM}$, $M\subset
Q\times\bR n$ that of $\mf M$ and $\xi:Q\to\bR n$ that of $\mf{\xi}$.
Then $\dot q\in \cM_q$ if and only if $\dot q= \xi(q)+u$ for some
$u\in\ker S(q)$, that is, if and only if  $S(q)[\dot q-\xi(q)]=0$.
Thus
$$
  M = \big\{ (q,\dot q)\in Q\times \bR n \,:\, S(q)\dot q + s(q) = 0 
  \big\}
$$
with
$$
   s(q) = -S(q) \xi(q) \in \bR{k} \,.
$$
Note that $s$ is independent of the arbitrariness in the choice of the
component of $\mf{\xi}$ along $\mf{\cD}$, that is of the component of
$\xi$ along $\ker S$. However, $s$ depends on the choice of the matrix
$S$: if $s_1$ and $s_2$ are relative to two matrices
$S_1$ and $S_2=P S_1$, then $s_2 =P s_1$. 

In coordinates, the equations of motion (\ref{EqLagrWithRF})
of the nonholonomic mechanical system $(\mf{L},\mf{Q},\mf{M})$ are
\begin{equation}\label{eq:LagrEq}
  \Big( \frac{d}{dt}\frac{ \partial L }{\partial \dot q} 
  - 
  \frac{\partial L}{\partial q}\Big) \Big|_M \, = \, R_{L,M}
\end{equation}
where $R_{L,M}(q,\dot q)$ is the local representative of
$\mf{R}_{\mf{L},\mf{M}}$. As shown in the Appendix, $R_{L,M}$ equals
the restriction to $M$ of the function
\begin{equation}\label{eq:RF1}
  S^T(SA^{-1}S^T)^{-1} ( SA^{-1} \ell - \s )
\end{equation}
where $\ell\in\bR n$ and $\s\in\bR k$ have components\footnote{We use
everywhere the convention of summation over repeated indexes.}
\begin{equation}\label{eq:RF2}
  \ell_i = \dder L {\dot q_i}{q_j}\dot q_j 
  - \der L {q_i}
  \,,\qquad 
  \sigma_a = 
  \der{S_{ai}}{q_j} \dot q_i\dot q_j + 
  \der{s_{a}}{q_j} \dot q_j 
\end{equation}
with $i,j,h=1,\ldots,n$ and $a = 1,\ldots, k$. Note that
$$
   \ell=\alpha+\beta+V'
$$
with
$
\alpha_i \, = \,
     \left( \frac{\partial A_{ij}}{\partial q_h}
           -\frac12 \frac{\partial A_{jh}}{\partial q_i}\right)
     \dot q_j\dot q_h
$,
$
\beta_i=\Big( \frac{\partial b_j}{\partial q_i}
                -\frac{\partial b_i}{\partial q_j} \Big)\dot q_j
$ and
$V'_i=\der V{q_i}$. 

In the case of linear constraints, expression (\ref{eq:RF1}) or
analogue expressions are given in \cite{ago,FRS,benenti}; in
the Appendix, besides considering the affine case, we complement these
treatments with a global perspective.

\begin{remark}\rm 
The restriction to $M$ of the function (\ref{eq:RF1}) is independent
of the choice of $S$ and $s$, that is, under the replacement of $S,s$
by $P S,P s$. (This change produces an extra term in $\sigma$, which
however vanishes on $M$).
\end{remark}

\subsection{The reaction-annihilator distribution}
While the condition of ideality assumes that, at each point $\mf
q\in\mf Q$,  the constraint can---a priori---exert all reaction
forces that lie in $\mf{\cD}_{\mf q}^\circ$, expression
(\ref{eq:RF1}) shows that, ordinarily, only a subset of these
possible reaction forces is actually exerted in the motions of the
system. In fact, in coordinates,  $\mf{\cD}_{\mf q}^\circ$ is the
orthogonal complement to $\ker S(q)$, namely the range of $S(q)^T$,
and the map
$$
   S^T(SA^{-1}S^T)^{-1} (SA^{-1} \ell - \s ) 
   \big|_{M_q} :M_q\to\mathrm{range} S(q)^T
$$
might not be surjective. 
Specifically, the reaction forces that the constraint exerts, when the
system $(\mf{L},\mf{Q},\mf{M})$ is in a configuration $\mf{q}\in
\mf{Q}$ with any possible
velocity $\mf{v}_{\mf q} \in \mf{\cM}_{\mf q}$, are the elements of the set
\[
  \mf{\cR}_{\mf q}:= 
  \bigcup_{\mf{v}\in \mf{\cM}_{\mf q}}
  \mf{R}_{\mf{L},\mf{M}}(\mf{v}_{\mf q})
\]
and this set may be (and typically is) a proper subset of
$\mf{\cD}_{\mf q}^\circ$.  This point of view was taken in \cite{FRS}
and leads to the following

\begin{definition}
The {\rm reaction-annihilator distribution} $\mf{\cR}^\circ$ of a
nonholonomic mechanical system with affine constraints
$(\mf{L},\mf{Q},\mf{M})$  is the (possibly non-smooth and of
non-constant rank) distribution on $\mf{Q}$ whose fiber
$\mf{\cR}^\circ_{\mf q}$ at $\mf{q}\in \mf{Q}$ is the annihilator of
$\mf{\cR}_{\mf q}$. 
\end{definition}

The interest of this distribution is not that much geometric, but
mechanical: a vector field $\mf{Z}$ on $\mf{Q}$ is a section of
$\mf{\cR}^\circ$ if and only if, {\it in all constrained motions} of
the system, the reaction force does zero work on it: $\langle \mf{R}_{
\mf{L},\mf{M}}  (\mf v_{\mf q}) ,\mf Z(\mf q) \rangle =0$ for all
$\mf{v}\in \mf{M}$. This is a system-dependent condition, which is
weaker than being a section of $\mf{\cD}$ because
$$
  \mf{\cD}_{\mf q} \subseteq \mf{\cR}_{\mf q}^\circ
  \qquad \forall \, \mf{q} \in \mf{Q} 
  \,.
$$

Expression (\ref{eq:RF1}) of the reaction forces shows that
$\mf{\cR}^\circ$ is computable a priori, without knowing
the motions of the system. Moreover, this expression shows how
$\mf{\cR}^\circ$ depends on the geometry of the nonholonomic constraints
(through the matrix $S$ and the vector~$s$) and on the mass distribution
of the system and on the active forces that act on it (through the
Lagrangian).\footnote{$\mf{\cR}^\circ$ depends of course also on the
holonomic constraint, through the dependence of $L$ on $(q,\dot q)$.}

Examples of reaction-annihilator distributions $\mf{\cR}^\circ$, that show
that their fibers may actually be larger than those of $\mf{\cD}$,
are given in \cite{FRS,FGS2008} for systems with linear constraints
and in sections 3.2 and 4.3 below for the case of affine
constraints. For  a discussion of the relation between $\mf{\cR}^\circ$ 
and d'Alembert principle see \cite{FS2009}. 

\section{Conservation of energy in systems with affine constraints}

\subsection{Characterization of the conditions for energy conservation}
We recall that the {\it energy}, or more exactly the {\it Jacobi
integral}, of a Lagrangian $\mf{L}$ is the function
$$
  \mf{E}_{\mf L}(\mf{v}) := \langle \mf{p}_{\mf L}, \mf{v} \rangle 
  - 
  \mf{L}(\mf v)
$$
where $\mf{p}_{\mf L}$ is the momentum 1-form relative
to $\mf{L}$, namely $\mf{p}_{\mf L} =\mathbf {F} L$
with $\mathbf F$ the fiber derivative (as defined, e.g., in
\cite{abraham-marsden}). In coordinates, $p_L=\der L{\dot q}$ and, if
$L$ is as in (\ref{L}), $p_L(q,\dot q)=A(q)\dot q - b(q)$. 

If $\mf L$  is as in (\ref{hatL}) then  $\mf{E}_L =
\mf{T}+\mf{V}\circ\pi$. The function $\mf{E}_{\mf L}$ can be
properly interpreted as the mechanical energy of the system only if
$\mf{L}=\mf{T}-\mf{V}\circ\pi$ but, as is customary, we will call it
energy in all cases. For a Lagrangian system, $\mf{E}_{\mf L}$ is a
first integral if and only if, as we do assume here, $\mf{L}$ is
independent of time.

\begin{definition}
{\it The {\rm energy} $\mf{E}_{\mf{L},\mf{M}}$ of the nonholonomic
mechanical system with affine constraints $(\mf{L},\mf{Q},\mf{M})$ 
is the restriction of $\mf{E}_{\mf L}$ to the constraint manifold
$\mf{M}$.
}\end{definition}

\begin{proposition}\label{thm1}
For a nonholonomic mechanical system with affine constraints 
$(\mf{L},\mf{Q},\mf{M})$ with constraint distribution
$\mf{\cD}+\mf{\xi}$, the energy $\mf{E}_{\mf{L},\mf{M}}$ is a first
integral if and only if $\mf{\xi}$ is a section of $\mf{\cR}^\circ$.
\end{proposition}

\begin{proof} The proof can be given in coordinates.
By (\ref{eq:LagrEq}), along any curve $t\mapsto (q_t,\dot q_t)\in M$,
\begin{equation}
\label{sec:lie-jacobi}
 \frac d{dt} E_{L,M}(q_t,\dot q_t)
 =
 \dot q_t\cdot 
 \left( \frac{d}{dt}\frac{\partial L}{\partial \dot q}
       - \frac{\partial L}{\partial q}\right)  (q_t,\dot q_t)
 =
 \dot q_t\cdot R_{L,M} (q_t,\dot q_t) \,.
\end{equation}
If $\dot q\in\cM(q)$, namely $\dot q =u+\xi(q)$ with $u\in \cD_q$, then
$$
  R_{L,M} (q,\dot q)\cdot \, \dot q = R_{L,M} (q,\dot q) \cdot \, \xi(q) 
$$
given that $R_{L,M} $ is ideal and hence annihilates $\cD_q$. 
It follows that $E_{L,M}$ is a first integral if and only if 
$R_{L,M} (q,\dot q) \cdot \xi(q) = 0$ for all $q\in Q$, $\dot q\in \cM_q$,
that is $\xi(q)\in \cR^\circ_q$ for all~$q\in Q$. 
\end{proof}

This shows that energy conservation is not a universal property of
nonholonomic mechanical systems with affine constraints. In particular, as we
have already stressed, it depends on the active forces that act on the
system.

In this respect we note that
at each point $\mf q\in \mf Q$, the
union of the fibers at $\mf q$ of the distributions  $\mf\cR^\circ$
relative to all functions $\mf V:\mf Q\to\bR{}$ equals $\mf\cD_q$. (In
fact, in expression (\ref{eq:RF1}), $\ell=\alpha+\beta+V'$ and the
matrix $S^T\big(SA^{-1}S^T\big)^{-1}SA^{-1}$ is, at each point~$q$,
the $A(q)^{-1}$-orthogonal projector onto $\cD_q^\circ$).
Thus, it follows from Proposition 1 that, given $\mf{T}$ and
$\mf{M}$, $\mf{E}_{\mf{L},\mf{M}}$ is a first integral in {\it all}
nonholonomic mechanical systems of the class
$(\mf{L}=\mf{T}-\mf{V}\circ\pi,\mf{Q},\mf{M})$, with {any} $\mf{V}:\mf
Q\to\bR{}$, if and only if the constraint is linear. This fact, which
can be generalized to allow for the presence of gyrostatic terms in
the Lagrangian, is a particular case of a result by 
\cite{terra-kobayashi} for general nonlinear constraints (with
linearity replaced by homogeneity). 

\begin{remark} {\rm
The result in Proposition \ref{thm1} does not depend on the choice of
$\xi$, which is defined up to the addition of a section of $\cD$,
because $\cD\subseteq \cR^\circ$. 
}\end{remark}

\subsection{Examples}

{\it 1. An affine nonholonomic particle. }
In order to illustrate the dependency of energy conservation on the
active forces we consider an affine version of Pars' nonholonomic
particle \cite{pars}. This system has configuration manifold $\bR3\ni
q= (x,y,z)$, Lagrangian
$$
  L(q,\dot q) = \frac12\|\dot q\|^2 - V(q)
$$
with a potential energy $V$ for which we will make different choices,
and constraint
$$
   \dot z +x\dot y -y\dot x-c =0
$$
with $c$ a nonzero real number. Thus  $n=3$ and $k=1$,
$\xi=c\partial_z$ and 
the distribution $\cD$ is spanned by the two vector fields 
$\partial_x+y\partial_z$ and $x\partial_x+y\partial_y$. $\cD$
has rank 2 except where $y=0$, so we disregard
these points and restrict the configuration manifold to
$Q=\bR3\setminus\{y=0\}$.
Note that not only $\xi$ is not a section of $\cD$ but, after the
restriction to~$Q$,
\begin{equation}\label{xi-not-in-D}
  \xi(q)\notin\cD_q \qquad \forall \, q\in Q \,.
\end{equation}
The constraint manifold $M$ is diffeomorphic to $Q\times\bR2$
and has global coordinates $(x,y,z,\dot x,\dot y)$. We may take
$S(x,y,z) = ( -y,x,1)$ and $s=-c$. 

Due to its low dimensionality, this example offers little variety of
behaviours. Specifically, the distribution $\cR^\circ$ depends on the
function $V$ but, since its fibers contain those of $\cD$, that have
dimension 2, there are only two possibilities: at a point $\bar q$,
either $\cR^\circ_{\bar q}= T_{\bar q}Q=\bR3$ or $\cR^\circ_{\bar
q}=\cD_{\bar q}$. The former possibility is realized if
$R_{L,M}(\bar q,\dot q)=0$ for all $\dot q\in \bR3$ and the second if
$R_{L,M}(\bar q,\dot q)\not=0$ for some $\dot q\in \bR3$.

At the same time, however, the low dimensionality makes 
all computations straightforward, and it is
simple to find potentials $V$ that exemplify the
different possibilities. In fact $A=\mathbb{I}$ and $\alpha=\beta=0$,
so $\ell=V'$, and $\sigma=0$ and (\ref{eq:RF1}) gives
\begin{equation}
\label{esempio}
   R_{L,M}
   \;=\;
   S^T (S S^T)^{-1} V' \Big|_M
   \;=\;
   \frac1{1+x^2+y^2}
   \begin{pmatrix}
       xy^2 & -xy & -y \\
       -xy  & x^2 &  x \\
       -y & x & 1 
  \end{pmatrix}
  \, V' \Big|_M \,.
\end{equation}
Thus, the reaction forces are independent of the $\dot q$. This
would make it straightforward to compute the fibers $\cR^\circ_q$,
which are simply the orthogonal complements to $R_{L,M}(q)$ in
$\bR3$, but we need not doing it because we already know what these
fibers are. 

Concerning the conservation of energy, there are two cases to consider:

\begin{list}{}
{\leftmargin2em\labelwidth1.2em\labelsep.5em\itemindent0em
\topsep0.5ex\itemsep-0.2ex}
\item[1.] If $V$ is such that $R_{L,M}(q)=0$ at all points $q\in Q$,
then $\cR^\circ=\bR3$ and energy is conserved.\footnote{That energy is
conserved if the reaction forces vanish identically is of course
obvious, for a variety of reasons. For instance, the nonholonomic
system is a subsystem of the unconstrained system with Lagrangian $L$
on $TQ$.} From (\ref{esempio}) one verifies that this situation is
encountered, e.g., with $V=0$ and $V=\frac12(x^2+y^2)$.

\item[2.] If $V$ is such that $R_{L,M}$ is not identically zero, then there
is a point $\bar q\in Q$ at which $\cR_{\bar q}^\circ=\cD_{\bar q}$.
Thus, by (\ref{xi-not-in-D}), 
$\xi(\bar q)\notin \cR_{\bar q}^\circ$ and so $\xi$ is not a section of
$\cR^\circ$. Hence energy is not conserved. An example is $V=z$. 
\end{list}
A richer variety of possibilities, including (a) conservation of
energy even with nonzero reaction forces, and (b) violation of the
conservation of energy even in the absence of active forces, can be
easily constructed by considering four- or five-dimensional
extensions of the nonholonomic particle, similar to that considered
in \cite{FGS2009}.  However, cases (a) and (b) are met also in known
mechanical systems. For instance, in the system formed by a heavy
sphere that rolls on a rotating horizontal plane, considered e.g. in
\cite{NF}, the potential energy
of the active forces is constant but the energy is not conserved
\cite{FS2015}. An example of case (a) is the following one.
  
\vskip4mm
\noindent{\it 2. A sphere rolling inside a rotating cylinder. }
As a second example we consider the system formed by a
homogeneous sphere constrained to roll without sliding inside a
cylinder that rotates with constant angular velocity about
its figure axis. We assume that the sphere is acted upon by
positional forces whose potential energy is a function of the position of
the sphere's center of mass. The case of a heavy sphere inside a
vertical cylinder that is at rest is classical \cite{routh, NF}. The
case of rotating cylinder, but without active forces, is considered
in~\cite{BMK2002}. 

Let $a$ be the radius of the sphere and $r+a$ the radius of the cylinder.
The holonomic system we start from
is formed by the sphere constrained to keep its center on a cylinder
$C$ of radius $r$. Fix an orthonormal frame $\Sigma=\{O;e_x,e_y,e_z\}$
with the origin $O$ on the figure axis of the cylinder and $e_z$
aligned with it. Using cylindrical coordinates
$(z,\gamma)$ relative to $\Sigma$ we identify $C$
with $\bR{}\times S^1$ and the configuration manifold is
$\mf{Q} = \bR{}\times S^1 \times \mathrm{SO}(3)\ni (z,\gamma,\cR)$,
where the matrix $\cR$ gives the sphere's orientation.
Up to an overall factor, the Lagrangian is then
\begin{equation}\label{L-cilindro}
  \mf L = \frac{1}{2}\left(r^2\dot\gamma^2+\dot z^2\right) + 
  \frac{I}{2}\|\omega\|^2   - V(z,\gamma)
\end{equation}
where (up to the same factor) $I$ is the moment of
inertia of the sphere and $V$ is the potential energy of the active
forces, and $\omega=(\omega_x,\omega_y,\omega_z)$ is the
angular velocity of the sphere relative to $\Sigma$ (that we think of as a
function on $T\mathrm{SO}(3)$). We add now the
nonholonomic constraint that the sphere rolls without sliding on a
cylinder, coaxial with $C$ and of radius $r+a$, that rotates with
constant angular velocity $\Omega e_z$ relative to $\Sigma$:
\begin{equation}\label{M-cilindro}
  r\dot\gamma+a\omega_z-(r+a)\Omega=0
  \,,\qquad 
  \dot z +a (\omega_x\sin\gamma- \omega_y\cos\gamma)=0  \,.
\end{equation}
These equations can be solved for $(\dot z,\dot \gamma)$ and the
constraint manifold $\mf{M}$ is thus diffeomorphic to $\bR{}\times
S^1\times T\mathrm{SO(3)}$. 

As local coordinates on $\mf{Q}$ we use the cylindrical coordinates
$(z,\gamma)$ of the center of mass of the sphere and three Euler angles
$(\varphi,\psi,\theta)\in S^1\times S^1\times(0,\pi)$ that fix the
orientation of a body frame relative to $\Sigma$ (we adopt the
convention of \cite{arnold-mmmc} for the choice of these angles). The
representative of $\mf L$ is 
$$
  L=\frac{1}{2}\left(r^2\dot\gamma^2+\dot z^2\right) + 
  \frac{I}{2}(\dot\theta^2+\dot\varphi^2+\dot\psi^2 + 
   2\dot\varphi\dot\psi\cos\theta)
  - V(z,\gamma) 
$$
and the constraint (\ref{M-cilindro}) becomes
$$
  r\dot\gamma+a(\dot\varphi+\dot\psi\cos\theta)
  - (r+a)\, \Omega =0
  \,,\qquad
  \dot z+a \,
  [\dot\psi\sin\theta\cos(\gamma-\varphi)+\dot\theta\sin(\gamma-\varphi)]
  = 0 \,.
$$
In these coordinates the vector field $\mf\xi$ becomes the constant
vector field $\xi=\Omega(\partial_\gamma+ \partial_\varphi)$, or $\xi
= (0,\Omega,\Omega,0,0)$, and a possible choice of $S$ and $s$ is
$$
   S = 
   \left(\begin{matrix}
     0 & r & a & a \cos\theta & 0 \\
     1 & 0 & 0 & a \sin\theta \cos(\gamma-\varphi) 
       & a \sin(\gamma-\varphi) 
   \end{matrix}\right)
   \,,\qquad
   s = (0,- (r+a) \,\Omega) \,.
$$
As local coordinates on $\mf{M}$ we may use
$(z,\gamma,\varphi,\psi,\theta,\dot\varphi,\dot\psi,\dot\theta)$. 
From (\ref{eq:RF1}), the reaction force is then
$$
 R_{L,M}
 =
 \frac{I}{I+a^2} \Big(
 f \,,\;
 V'_\gamma \,,\;  
 \frac{a}{r} V'_\gamma \,,\;  
 af \cos(\gamma-\varphi) \sin\theta + \frac{a}{r}V'_\gamma\cos\theta \,,\;  
 af \sin(\gamma-\varphi) 
 \Big)
$$
with
$
  f 
  = 
  \frac{a^2}r 
  \big(\dot\varphi+\dot\psi\cos\theta - \frac{a+r}a \Omega \big)
  \big( \dot\theta \cos(\gamma-\varphi) - 
        \dot\psi \sin(\gamma-\varphi) \sin\theta \big) 
  + V'_z 
$.
Therefore, 
$$
  R_{L,M} \cdot \xi = \frac{a+r}a \Omega V'_\gamma \,.
$$
This shows that, when $\Omega\not=0$ and hence $\xi\not=0$ and
the constraint is affine,
the energy is conserved if and only $V$ depends on $z$ alone. This
includes the case of a heavy sphere that rolls inside a rotating vertical
cylinder, for which $V=gz$.

\section{Conservation of momenta and gauge momenta of lifted actions}

\subsection{Conservation of momenta}
We consider now a second problem in which the reaction forces of a
nonholonomic mechanical system with affine constraints play a
role: the conservation of the momentum map of a lifted action that
leaves the Lagrangian $\mf L$ invariant, and of its `gauge'
generalization. 

This topic has been widely studied in the case of nonholonomic
mechanical systems
with linear constraints. For such systems, the momentum map is in
general not conserved, but in certain cases some of its components
are conserved. In early studies, it was pointed out that a sufficient
condition for the conservation of a component of the momentum map is
that its infinitesimal generator is `horizontal', that is, a section
of the constraint distribution (see e.g.
\cite{bates-sniatycki,marle95,BKMM}). It was later proved that the
components of the momentum map that are conserved are exactly those
whose infinitesimal generators are sections of $\cR^\circ$ \cite{FRS}.
Our first goal here is to show that this result holds in the case of
affine constraints, too. In fact, the affine part of the constraint
plays no role in it.

Consider an action $\mf{\Psi}:G\times \mf{Q} \to \mf{Q}$ of a Lie
group $G$ on the configuration manifold  $\mf{Q}$. For each
$\mf{q}\in\mf{Q}$ we write as usual
$\mf{\Psi}_{g}(\mf{q})$ for $\mf{\Psi}(g,\mf{q})$.
The tangent lift $\mf{\Psi}^{T\mf{Q}}:G\times
T\mf{Q}\to T\mf{Q}$ of the action $\mf{\Psi}$ is the action of $G$ on
$T\mf{Q}$ given by
$$
   \mf{\Psi}^{T\mf{Q}}_{g}(\mf{v}_{\mf q}) = T_{\mf q}\mf{\Psi}_g
  \cdot \mf{v}_{\mf q} 
$$
(in coordinates, $\Psi^{TQ}_g(q,\dot q) = \big( \Psi_g(q),
\Psi'_g(q)\dot q\big)$ with $\Psi'_g = \der{\Psi_g}q$). 
We denote by $\mf{Y}_\eta := \frac d{dt}\mf{\Psi}_{\exp(t\eta)}|_{t=0}$ the
infinitesimal generator relative to an element 
$\eta\in\mathfrak g$, the Lie algebra of
$G$. Correspondingly, the $\eta$-component of the momentum map of
$\mf{\Psi}^{T\mf{Q}}$ is the function $\mf{J}_\eta :T\mf Q\to\bR{}$
defined as
$$
  \mf{J}_\eta (\mf v_{\mf q}) := 
  \langle \mf{p}_{\mf L}(\mf v_{\mf q}) , \mf{Y_\eta}(\mf q) \rangle
$$
(in coordinates, $\der L{\dot q}\cdot Y_\eta$). 

The tangent lift of a vector field $\mf{Z}$ on $\mf{Q}$ is the 
vector field $\mf{Z}^{T\mf{Q}}$ on $T\mf{Q}$ whose integral curves
$t\mapsto \mf{v}(t)$ are velocities of integral curves $t\mapsto
\mf{q}(t)$ of $\mf{Z}$, that is $\mf{v}(t)=\mf{Z}(q(t))\in
T_{\mf{q}(t)}\mf{Q}$ (in coordinates,
$Z^{TQ} 
 = 
 Z_i\partial_{q_i} + \dot q_j \der{Z_{i}}{q_j} \partial_{\dot q_i}$).
Clearly, $\mf{Y}_\eta^{T\mf{Q}}
=\frac d{dt}\mf{\Psi}^{T\mf{Q}}_{\exp(t\eta)}|_{t=0}$.
  
Consider now a nonholonomic mechanical system with affine constraints 
$(\mf L,\mf Q,\mf M)$ and assume that $\mf L$ is invariant
under $\mf\Psi^{T\mf Q}$,
namely $\mf{L}\circ\mf{\Psi}^{T\mf{Q}}_{g}=\mf{L}$ for all $g\in G$.
Then, we say that the function
$
 \mf{J}_\eta \big|_{\mf M} 
$
is the {\it momentum} of $(\mf L,\mf Q,\mf M)$ generated by $\mf
Y_\eta$. 

\begin{proposition}\label{prop4} Assume that $\mf L$ is invariant
under $\mf\Psi^{T\mf Q}$. Then a momentum
is a first integral of  $(\mf{L},\mf{Q},\mf{M})$ if and only if its
generator is a section of $\mf{\cR}^\circ$.
\end{proposition}

\begin{proof} We may work in coordinates. A computation gives
$\frac d{dt} (J_\eta|_M) = Y_\eta^{TQ}(L)\big|_M + R_{L,M}\cdot Y_\eta
\big|_M$. The invariance of $L$ implies $Y_\eta^{TQ}(L)=0$. Thus,
$J_\eta|_M$ is a first integral if and only if, at each $q\in Q$,
$Y_\eta$ annihilates all reaction forces $R_{L,M}(q,\dot q)$ with
$\dot q\in \cM_q$, that is, $Y_\eta(q)\in\cR^\circ_q$. \end{proof}

\subsection{Conservation of gauge momenta}
It was an original idea of \cite{BGM} that, for nonholonomic
mechanical systems with linear constraints whose Lagrangian is
invariant under a lifted action, certain conserved quantities that
are not components of the momentum map may be viewed as linked to the
group by a gauge-like mechanism. This situation extends to
nonholonomic mechanical systems with affine constraints. 
For a general study of this topic in systems with linear
constraints, and more information on the topic, including e.g. its
relation to the so called `momentum equation', see
\cite{FGS2008,FGS2009,FS2009,FGS2012}. 

Following the terminology of \cite{FGS2009} we say that a vector field
$\mf Y$ on $\mf Q$ is a {\it gauge symmetry} of
$(\mf L,\mf Q,\mf M)$ relative to the action $\mf \Psi$ if it 
is everywhere tangent to the orbits of $\mf \Psi$ and, moreover,
$$
  \mf Y^{T\mf Q}(\mf L)\big|_{\mf M}=0 \,.
$$
This invariance condition of $\mf L$ is independent of the
invariance of $\mf L$ under $\mf\Psi^{T\mf Q}$, even though it implies
that $\mf V$ is $\mf\Psi$-invariant. 
The {\it gauge momentum} generated by a gauge symmetry $\mf Y$ is the
function 
$$
  \mf J := 
  \langle \mf{p}_{\mf L} , \mf{Y} \rangle \big|_{\mf M} \,.
$$
The gauge symmetry that generates a given gauge momentum needs not be
unique. 

\begin{proposition}
A gauge momentum is a first integral of $(\mf L,\mf Q,\mf M)$ if and
only if it is generated by a gauge symmetry which is a section of
$\cR^\circ$.
\end{proposition}

The proof goes just as that of Proposition \ref{prop4}. 

The need of considering gauge momenta generated by gauge symmetries
that are sections of $\cR^\circ$, not only those generated by sections
of $\cD$, is demonstrated by the example in the following section.
This is the heavy sphere that rolls inside a rotating vertical 
cylinder. This system has as a first integral that depends smoothly
on the angular velocity of the cylinder and can be interpreted as a
gauge momentum. Interestingly, when the cylinder is at rest this
gauge momentum is generated by a gauge symmetry that is a section of
$\cD$; but as soon as the cylinder rotates, the generating gauge
symmetry leaves $\cD$ and becomes a section of $\cR^\circ$.


\subsection{Example}
Consider a heavy sphere that rolls without sliding inside a cylinder
that rotates uniformly about its vertical axis, namely, the system of
section 3.2 with potential energy $V=gz$.
The Lagrangian $\mf L$ and the configuration
manifold $\hat Q$ are independent of $\Omega$, while the constraint
manifold $\mf M$ depends on $\Omega$ through the vector field
$\mf\xi$. We thus denote it by $\mf M_\Omega$. 

It is classically known \cite{routh,NF} that, when $\Omega=0$, the system 
$(\mf L,\mf Q,\mf M_0)$ has the two first integrals
$\mf F_0:=\mf F|_{\mf M_0}$ and $\mf K_0:=\mf K|_{\mf
M_0}$, where
\begin{equation}\label{FeH}
  \mf F = I\omega_z -ar \dot\gamma
  \,,\qquad
  \mf K = a ( \omega_x \cos\gamma+ \omega_y \sin\gamma) -z\dot\gamma 
  \,.
\end{equation}
These two first integrals have been linked in \cite{BGM} to the action
$\mf\Psi$ of the group $G=S^1\times\mathrm{SO(3)}\ni(\zeta,S)$ on $\mf
Q\ni(z,\gamma,\cR)$ given by
$
  \mf\Psi_{\zeta,S}(z,\gamma,\cR) = (z, \gamma+\zeta, S_\zeta \cR S)
$,
where $S_\zeta$ is the matrix of the rotation by $\zeta$ about the
third axis. This action emerges naturally in this problem because it
leaves the Lagrangian (\ref{L-cilindro}) and the constraint
(\ref{M-cilindro}) invariant.
With reference to this action, $\mf F_0$ is a momentum
generated by an infinitesimal generator that is a section of $\mf\cD$
and $\mf K_0$ is a gauge momentum generated by a gauge symmetry which
is a section of~$\mf\cD$, see \cite{BGM}. 

When $\Omega\not=0$, the system  $(\mf L,\mf Q,\mf M_\Omega)$
has the two first integrals
$$
  \mf F_\Omega:=\mf F|_{\mf M_\Omega} \,,\qquad 
  \mf K_\Omega:=\mf K|_{\mf M_\Omega}
$$
which appear in \cite{BMK2002}. (Reference \cite{BMK2002} considers
only the case $V=0$, but these two first integrals exist for any
$V=V(z)$, see the remark below). We now show that, with reference to
the considered action $\mf \Psi$, when $\Omega\neq0$ the function
$\mf F_\Omega$ is a momentum generated by an infinitesimal generator
which is a section of $\mf\cD$ and the function $\mf K_\Omega$ is a
gauge momentum generated by a gauge symmetry which is a section of
$\mf \cR^\circ$, not of $\mf \cD$. Therefore, at variance from the
case $\Omega=0$, in order to link the first integral $\mf K_\Omega$
to the group action using the gauge mechanism, when $\Omega\neq0$
it is necessary to take into account the role of the reaction forces.

To prove these assertions we pass to the local coordinates
$(z,\gamma,\varphi,\psi,\theta)$ on $\mf Q$. The 
representatives $F$ of $\mf F$ and $K$ of $\mf K$ are obtained from
(\ref{FeH}) with
$\omega_x = \dot\theta\cos\varphi+\dot\psi\sin\varphi \sin\theta$, 
$\omega_y = \dot\theta\sin\varphi-\dot\psi\cos\varphi \sin\theta$, 
$\omega_z =  \dot\varphi+\dot\psi \cos\theta$. 
It follows from the analysis of section 3.2 that, in these
coordinates, 
$$
\begin{aligned}
  &\cD =
   \textrm{span}_\bR{}
   \big\{
   a\partial_\gamma - r \partial_\varphi  , 
   \partial_\theta - a \sin(\gamma-\varphi) \,\partial_z ,
   \partial_\psi -a \cos(\gamma-\varphi) \sin\theta\,\partial_z -
        \cos\theta\,\partial_\varphi 
    \big\} 
  \\
  &\cR^\circ =
  \textrm{span}_\bR{}  
  \left\{
  \partial_\gamma \,,\; \partial_\varphi \,,\; 
  \partial_\theta - a \sin(\gamma-\varphi) \,\partial_z \,,\; 
  \partial_\psi -a \cos(\gamma-\varphi) \sin\theta\,\partial_z -
  \cos\theta\,\partial_\varphi 
  \right\} \,.
\end{aligned}
$$
The tangent spaces to the group orbits are spanned by $\partial_\gamma$
and by three infinitesimal generators of the
$\mathrm{SO(3)}$--action, e.g. the generators 
\[
\begin{aligned}
   &\eta_x =
   \sin\varphi (\partial_\psi-\cos\theta\,\partial_\varphi)
   +\cos\varphi\, \sin\theta \,\partial_\theta 
   \,, \qquad
   \\ 
   &\eta_y = 
   \cos\varphi (\partial_\psi-\cos\theta\,\partial_\varphi)
   - \sin\theta \,  \sin\varphi \, \partial_\theta \,,
  \\
  &\eta_z 
  = 
  \partial_\varphi 
\end{aligned}
\]
of the components $\omega_x,\omega_y,\omega_z$ of the
$\mathrm{SO(3)}$-momentum map. The vector fields
$$
  Y_F:=\eta_z-\frac ar\partial_\gamma
  \,,\qquad
  Y_K := 
  \frac{a}{I \sin\theta} (\eta_x \cos\gamma - \eta_y\sin\gamma) 
  -\frac{z}{r^2} \partial_\gamma 
$$
are tangent to the group orbits. $Y_F$ is an infinitesimal generator
of the $S^1\times\mathrm{SO}(3)$-action and is a section of $\cD$. As
such it generates a conserved momentum, that equals $F_\Omega$. 
$Y_K$ is instead a section of $\cR^\circ$ and is 
(the local representative of) a gauge symmetry because 
$$
  Y^{TQ}(L) =
  \big(\dot z + a\omega_x\sin\gamma- a\omega_y\cos\gamma\big)
  \dot\gamma
$$
vanishes on $M_\Omega$, see (\ref{M-cilindro}). 
Thus, $Y_K$ generates a conserved gauge momentum,
which equals $K_\Omega$. 

If we use coordinates $(z,\gamma,\varphi,\psi,\theta,\dot
\varphi,\dot\psi,\dot\theta)$ on $\mf M_\Omega$, then 
$$
  F_\Omega = (I+a^2)\omega_z - a (a+r) \Omega
  \,,\qquad
  K_\Omega = a\omega_x\cos\gamma +  a\omega_y\sin\gamma +
             \frac ar z\omega_z - \frac{r+a}r \Omega z \,.
$$

It remains to prove that $K_\Omega$ is not generated by any gauge
symmetry which is a section of $\cD$. To this end, we make the
following observations. We call {\it generator} of a gauge momentum
any vector field $Z$---not necessarily a gauge symmetry---such that
$\mf J = \langle \mf{p}_{\mf L} , \mf{Z} \rangle \big|_{\mf M}$.
Then, {\it a gauge momentum has at most one generator that is a
section of $\cD$}. We may prove this in coordinates. If $W$ and $Z$
are the representatives of two generators of a gauge momentum $\mf J$
then $(W-Z)\cdot (A\dot q-b)|_M=0$. Equivalently, at each point $q$,
$(W-Z)\cdot (Au+A\xi-b)=0$ for all $u\in \cD_q$. Hence $Z$ and $W$
satisfy the two conditions
$$
   (W-Z)\cdot Au=0 \ \forall u\in \cD_q 
   \,,\qquad 
   (W-Z)\cdot (A\xi-b)=0 \,.
$$
Since $A$ defines a metric, the first of these two conditions implies
that, at each point $q$, $W-Z$ is orthogonal, in this metric, to
$\cD_q$. Hence all generators of a gauge momentum have the same
component along $\cD$. This argument shows that, moreover, {\it if $Y$ is
a generator of a gauge momentum, then its unique generator which is a
section of $\cD$, if it exists, is $\Pi_AY$}, where, at each point
$q$, $\Pi_A$ is the $A$-orthogonal projector onto $\cD_q$.
Furthermore, {\it $\Pi_AY$ is a generator of the gauge momentum if and only
if $(\Pi_AY-Y)\cdot(A\xi-b)=0$.}

In our case
$
  \Pi_AY_K 
  = 
  -\frac{a^2 z}{(a^2+I) r^2} \, \partial_\gamma 
  +
  \big( \frac{a z}{(a^2+I) r}+\frac{a\, \cos \theta
  \sin(\gamma-\varphi)}{ I\sin\theta} \big)\, \partial_\varphi
  - 
  \frac{a \sin(\gamma-\varphi)}{I \sin\theta} \, \partial_\psi
  +
  \frac aI \cos(\gamma - \varphi) \, \partial_\theta
$
is not a generator of $K_\Omega$ since $(\Pi_AY_K-Y_K)\cdot A\xi
= \frac{I (a+r) \Omega z}{(a^2+I) r} $.

\begin{remark} \rm
We have considered the case of constant gravity,
with potential energy $V(z)=gz$, in order to make a direct comparison
with \cite{BGM}, that treats only this case. However,
it is easy to verify that the two first integrals $F_\Omega$ and
$K_\Omega$ exist, and retain their interpretations as (gauge)
momenta, with the same generators, if the sphere is acted upon by
{\it any} potential energy that depends only on $z$, that is, which
is invariant under the considered action $\hat\Psi$. This is a
`Noether-like' property, called `weak-Noetherianity' in
\cite{FGS2008}. This example thus shows that, in this respect, the
case of nonholonomic systems with affine constraints differs from
that of nonholonomic systems with linear constraints, where the only gauge momenta
with such a weakly-Noetherian property are those that have a
generator which is a section of $\cD$ \cite{FGS2012}.
\end{remark}

\section{Conclusions}

At the core of Proposition 1 lies the balance equation of the energy
(\ref{sec:lie-jacobi}). Together with the assumption of 
ideality of the constraints it implies that
conservation of energy is equivalent to
\begin{equation}\label{ultima}
  \langle \mf R(\mf v),\mf\xi(\mf v)\rangle = 0 
  \qquad \forall v\in \mf M \,,
\end{equation}
that is, to the fact that, in any constrained kinematical state, the
reaction force $\mf R$ does not do any work on the nonhomogeneous
term $\mf\xi$ of the constraint.

If the only knowledge assumed on the reaction forces is d'Alembert
principle, namely that they annihilate the distribution $\mf\cD$, then
all can be deduced from equation (\ref{ultima}) is that
there is conservation of energy when $\mf\xi$ is a section of $\mf\cD$,
that is, when the constraint is linear. This point of view is quite
widespread: the very use of expressions such as ``undetermined
multipliers'' or ``unknown [reaction] forces'' indicates a
perception of the reaction forces as intrinsically unknown
objects.

The novelty of our approach is that we exploit the fact that, for a given
system, the reaction forces are known functions of the kinematic
states. This allows to identify the cases in which the energy is
conserved even if the constraint is genuinely affine. 

Very  similar considerations can be made concerning the conservation
of momenta and gauge momenta.

The general message we convey, besides the specific information about
the conservation of energy and (gauge) momenta in nonholonomic
mechanical systems with
affine constraints, is that the reaction forces should be taken into
a clearer account in theoretical studies of these systems.
Particularly when analyzing differences from holonomic systems 
(Noether theorem, conservation of energy, hamiltonianization,
perhaps invariant measures) it might be important to exploit the
presence of the reaction
forces in the equations of motion. From this point of view,
advancement in the comprehension of nonholonomic mechanical systems is not a purely
geometric matter and the comprehension of some dynamical aspects
might pass through a better
understanding---and perhaps attempts of classification---of the
reaction forces.

\section{Appendix: The equations of motion and the reaction forces}

We derive here the expression of the reaction forces as
functions of the kinematic states of the system for a nonholonomic
system with affine constraints. 

\begin{proposition} In the hypotheses and with the notation of section
2.1:

\begin{list}{}
{\leftmargin2em\labelwidth1.2em\labelsep.5em\itemindent0em
\topsep0.5ex\itemsep-0.2ex}
\item[1.] For any $\bar v_0\in \mf M$ there exist unique
curves $\bR{}\ni t\mapsto (\mf v_t,\mf R_t) \in \mf M\times \mf \cD^\circ$
such that  $\mf v_0=\bar v_0$ and
\begin{equation}\label{EqLagr}
  [\mf L](\mf v_t)  = \mf R_t  \qquad \forall t \,.
\end{equation}

\item[2.] There exists a function $\mf R_{\mf L,\mf M}: \mf M \to \mf
\cD^\circ$ such that, for all $\bar v_0$, $\mf R_t=\mf R_{\mf L,\mf
M}(\mf v_t)$ $\forall t$.

\item[3.] In any system of bundle coordinates $(q,\dot
q)$ in $T\mf Q$, if $M$ is described as $S(q)\dot q+s(q)=0$, then
the representative $R_{L,M}$ of $\mf R_{\mf L,\mf M}$ is the restriction to
$M$ of the function (\ref{eq:RF1}),~(\ref{eq:RF2}).

\end{list}
\end{proposition}

\begin{proof} We first work in coordinates  and then globalize the
result. The matrix $A=\dder L {\dot q}{\dot q} =
\dder {L_2} {\dot q}{\dot q}$ is nonsingular and independent of the
velocities. 

Assume $t\mapsto (q_t,\dot q_t,R_t)$ is the local representative of a
curve in $M\times \cD^\circ$ that satisfies (\ref{EqLagr}). Then, for
all $t$,
\begin{equation}\label{vincolo}
   S(q_t)\dot q_t+s(q_t) = 0 
\end{equation}
and $R_t\in\mathrm{range} S(q_t)^T$. Thus, there exists a curve
$t\mapsto \lambda_t\in\bR k$ (the Lagrange multiplier) such that
$R_t=S(q_t)^T\lambda_t$. Since  $[L]=A(q)\ddot q + \ell(q,\dot q)$,
equation (\ref{EqLagr}) is
$$
   A(q_t) \ddot q_t + \ell(q_t,\dot q_t) = S(q_t)^T\lambda_t 
$$
and gives $\ddot q_t = A(q_t)^{-1}[S(q_t)^T\lambda_t - \ell(q_t,\dot
q_t)]$.
Inserting this expression into $S(q_t)\ddot q_t +\s(q_t,\dot q_t)=0$,
that follows from (\ref{vincolo}), gives
$$
   S(q_t)A(q_t)^{-1}[S(q_t)^T \lambda_t - \ell(q_t,\dot q_t)] 
   + \s(q_t,\dot q_t)=0 \,.
$$
This  equation can be solved for $\lambda_t$ because the matrix
$SA^{-1}S^T$ is invertible and gives
\begin{equation}\label{lambda}
   \lambda_t = [S(q_t)A(q_t)^{-1}S(q_t)^T]^{-1}
   \big[ S(q_t)A(q_t)^{-1}\ell(q_t,\dot q_t) - \s(q_t,\dot q_t)\big] 
\end{equation}
or $R_t = R_{L,M}(q_t,\dot q_t)$ with $R_{L,M}$ as stated. Together
with the independence of $R_{L,M}$ of the choice of $S$ and $s$, see
the remark at the end of section 2.2, this proves that if $\mf R_{\mf
L,\mf M}$ exists, then its local representative is $R_{L,M}$ and hence
the uniqueness of $t\mapsto (q_t,\dot q_t,R_t) \in M\times \cD^\circ$, as in
item 1. 

Consider now the equation 
$$
  A\ddot q + \ell =  S^T(SA^{-1}S^T)^{-1} (SA^{-1} \ell - \s) 
$$
in $\bR n$. Let $t\mapsto q_t$ be its unique solution with initial datum
$(q_0,\dot q_0)\in M$. Define $t\mapsto \lambda_t$ as in
(\ref{lambda}). With this choice of $\lambda_t$, the curve
$t\mapsto q_t$ satisfies $S\ddot q_t+\s(q_t,\dot q_t)=0$, or 
$\frac d {dt}[S(q_t)\dot q_t+s(q_t)]=0$. Hence $(q_t,\dot q_t)\in M$
for all $t$ and satisfies equation (\ref{EqLagr}) with
$R_t=S(q_t)^T\lambda_t\in \cD^\circ_{q_t}$. By d'Alembert principle,
such an $R_t$ is the representative of a reaction force that the
nonholonomic constraint can exert. 

This proves a local version of items 1. and 2., which
hold true within each coordinate chart. 
 
We now globalize these results. To this end we verify that the local
representatives of equation (\ref{EqLagr}), with reaction forces that
have the expression (\ref{eq:RF1}), 
match in the intersection of different chart domains. Let $\tilde q\mapsto
q=\cC(\tilde q)$ be a change of coordinates in $Q$. The local
representatives of the affine subbundle $M$ in the new coordinates is still
given by an equation of the form $\tilde S(\tilde q)\dot{\tilde q} +
\tilde s(\tilde q)=0$ and a simple computation shows that
$$
  \tilde S = \tilde P\,[S\circ\cC]\, \cC' \,,\qquad 
  \tilde s = \tilde P\, [s\circ\cC] 
$$
where $\cC'$ is the Jacobian matrix of $\cC$ and $\tilde P=\tilde
P(\tilde q)$ is a $k\times k$ nonsingular matrix. The local
representatives $\tilde L$ of the Lagrangian in the new coordinates
is $\tilde L(\tilde q,\dot{\tilde q}) = L(\cC(\tilde q),\cC'(\tilde
q)\dot{\tilde q})$ and the matrix
$
  \tilde A 
  := 
  \dder {\tilde L}{\dot{\tilde q}}{\dot{\tilde q}}
$ is given by 
$$
  \tilde A  = {\cC'}^T\,[A\circ \cC]\, \cC ' \,.
$$
In order to prove the statement it suffices to show that, if
$t\mapsto q_t$ is a solution of
$$
    \frac d{dt} \der {L} {\dot q} - \der{L}q 
    =
    S^T(SA^{-1}S^T)^{-1} (SA^{-1} \ell - \s) \,, 
$$
then $t\mapsto \tilde q_t=\cC^{-1}(q_t)$ is a solution of 
$$
    \frac d{dt} \der {\tilde L} {\dot{\tilde q}} -
    \der{L}{\tilde q} 
    =
    \tilde S^T(\tilde S \tilde A^{-1} \tilde S^T)^{-1}
    (\tilde S \tilde A^{-1} \tilde \ell -  \tilde \s) 
$$
with $\tilde \ell$ and $\tilde \s$ defined as $\ell$ and $\s$ in
(\ref{eq:RF2}), but in terms
of $\tilde L$, $\tilde S$ and $\tilde s$. It is well known from 
Lagrangian mechanics that 
$$
    \frac d{dt} \der {\tilde L} {\dot{\tilde q}}
    - \der{\tilde L}{\tilde q} 
    =
   \cC '^T  
   \Big[
   \Big( \frac d{dt} \der {L} {\dot q} - \der{L}q \Big)\circ\cC
   \Big] \,.
$$
Elementary computations show that 
\begin{equation}
\begin{aligned}
  &\tilde \ell_i 
  = 
  ({\cC'}^T \ell)_i + 
  ({\cC'}^T A)_{ij} [\dot{\tilde q}\cdot \cC''_j\dot{\tilde q} ] 
  \\
  &\tilde \sigma_a 
  = 
  (\tilde P\sigma)_a + 
  (\tilde P S)_{aj} [\dot{\tilde q}\cdot \cC''_j\dot{\tilde q} ]
\end{aligned}
\end{equation}
where $\cC''_j$ is the Hessian matrix of the $j$-th component $\cC_j$
of $\cC$ and 
with the convention, used below as well, that $\ell$, $A$, $\sigma$
and $S$ are composed with $\cC$. Since
$\tilde S\tilde A^{-1}=\tilde P SA^{-1}{\cC'}^{-T}$, this implies
$
  \tilde S\tilde A^{-1}\tilde \ell - \tilde \s 
  =
  \tilde P[SA^{-1}\ell - \s ]
$
so that
$$
   \tilde S^T\big(\tilde S\tilde A^{-1}\tilde S^T\big)^{-1}
   \big[ \tilde S \tilde A^{-1} \tilde \ell - \tilde \s \big]
   =
   {\cC'}^T 
   \Big[ \big(S^T\big(SA^{-1}S^T\big)^{-1} \big[  SA^{-1} \ell + \s
           \big]\big) \circ \cC\Big] \,. 
$$
This completes the proof. \end{proof}


\end{document}